\documentclass[letterpaper,10pt]{amsart}
\usepackage{amsmath,amssymb,amsxtra,amsthm, amstext, amscd,amsfonts,fancyhdr, hyperref,enumerate,tikz}
\usepackage[all,cmtip]{xy}
\usepackage{tikz-cd}

\usepackage[OT2,T1]{fontenc}
\DeclareSymbolFont{cyrletters}{OT2}{wncyr}{m}{n}
\DeclareMathSymbol{\Sha}{\mathalpha}{cyrletters}{"58}

\newcommand{\bG}{{\mathbb{G}}}

\newcommand{\bQ}{{\mathbb{Q}}}

\newcommand{\bZ}{{\mathbb{Z}}}

  \newcommand{\A}{{\mathcal{A}}}
  
  \newcommand{\C}{{\mathcal{C}}}

  \newcommand{\G}{{\mathcal{G}}}
  \newcommand{\I}{{\mathcal{I}}}

\renewcommand{\L}{{\mathcal{L}}}

\renewcommand{\O}{{\mathcal{O}}}
\renewcommand{\P}{{\mathcal{P}}}

\renewcommand{\S}{{\mathcal{S}}}

\newcommand{\fp}{\mathfrak{p}}

\newcommand{\fm}{\mathfrak{m}}
\newcommand{\fn}{\mathfrak{n}}

\newcommand{\Hom}{\operatorname{Hom}}

\newcommand{\spec}{\operatorname{Spec}}

\newcommand{\fD}{\mathfrak{D}}
 
\newcommand{\iden}{\operatorname{id}}

\newcommand{\Ext}{\textnormal{Ext}}

\newcommand{\upchi}{{\raise.35ex\hbox{$\chi$}}}

\newcommand{\ra}{\rightarrow}

\newcommand{\QC}{\operatorname{QC}_\infty}
\newcommand{\Dbul}{\textnormal{D}^\bullet}
\newcommand{\Db}{\textnormal{D}^{\textnormal{b}}}
\newcommand{\dua}{\mathfrak{D}}
\newcommand{\eval}{\operatorname{eval}}

\makeatletter
\@namedef{subjclassname@2010}{%
  \textup{2010} Mathematics Subject Classification}
\makeatother

\newtheorem{theorem}{Theorem}[section]
\newtheorem{corollary}[theorem]{Corollary}
\newtheorem{proposition}[theorem]{Proposition}
\newtheorem{lemma}[theorem]{Lemma}

\theoremstyle{definition}
\newtheorem{definition}[theorem]{Definition}

\newtheorem{example}[theorem]{Example}
\newtheorem{remark}[theorem]{Remark}

\newtheorem*{theorem*}{Theorem}

\newtheorem*{proof of (1)}{proof of (1)}
\newtheorem*{proof of (2)}{proof of (2)}

\newtheorem*{mainTheoremA}{Main theorem}

\newtheorem{informal question}[theorem]{Informal Question}

\usepackage{color}
\newif\ifhascomments \hascommentstrue
\ifhascomments
  \newcommand{\brett}[1]{{\color{blue}[[\ensuremath{\textnormal{Brett}: } #1]]}}
  
\else
  \newcommand{\brett}[1]{}
  \newcommand{\Aji}[1]{}

\fi

\numberwithin{equation}{section}

\begin{document}

\title{Mukai Duality for abelian stacks}

\author[A. Dhillon]{Ajneet Dhillon}
\address{Department of Mathematics, Western University, Canada}
\email{adhill3@uwo.ca}

\author[B. Nasserden]{Brett Nasserden}
\address{Department of Mathematics, Western University, Canada}
\email{bnasserd@uwo.ca}

\maketitle

\begin{section}{Introduction}
As is well known \cite{FM}, for an abelian variety $A$ and its dual $A^t$ there is an equivalence of derived categories
\[D^b(A)\cong D^b(A^t)\]
given by a Fourier-Mukai transform with kernel the normalized Poincare bundle on $A\times A^t$. In this article, we will show that an analogous statement holds for tame abelian stacks. In \cite{Brochard1} abelian stacks were introduced to construct albanese morphisms of algebraic stacks; they may be thought of as stacky generalizations of abelian varieties. Every abelian stack $\A$ has a dual $\dua(A)$ and the product $\A\times \dua(\A)$ has a well-defined Poincare bundle. We first show that tame abelian stacks are perfect in the sense of \cite{bfn} in \ref{prop:abStPerfect} and conclude that $\QC(\A)$ (the stable $\infty-$category of quasi-coherent sheaves on $\A$) is self-dual by \cite[corollary 4.8]{bfn}. We then explicitly verify that $\QC(\dua(\A))$ is dual to $\QC(A)$ by a direct computation with the Fourier-Mukai transform induced by the Poincare bundle on $\A\times \dua(\A)$. The benefit of this approach is that it is equivalent to showing that a composition of morphisms
\[\QC(\A)\ra \QC(\A)\times \QC(
\dua(\A))\times \QC(\A)\ra \QC(\dua(\A))\]
is the identity, and this calculation can be made in a way analogous to Mukai's computation in \cite{FM}.

Every abelian stack is a strict commutative group stack in the sense of Deligne, see \cite[XVIII 1.4]{SGA4Tome3} where they are called Picard stacks. Commutative group stacks can be thought of as an enlargement of the category of sheaves of commutative groups, while algebraic commutative group stacks can be thought of as an enlargement of the category of commutative group schemes. Deligne's construction provides two alternative ways of thinking of commutative group stacks. Firstly, one may think of them as a stack $\G$ along with an addition law $m\colon \G\times \G\ra \G$ that satisfies the natural 2-categorical analogues of the commutative group scheme axioms. Alternatively, one may work with two-term complexes of sheaves of abelian groups $[G^{-1}\ra G^0]$. The correspondence
$[G^{-1}\ra G^0]\mapsto [G^0/G^{-1}]$
provides an equivalence of categories between two-term complexes of abelian sheaves and commutative group stacks on $S$. Here $[G^0/G^{-1}]$ is the quotient stack. See \cite[Theorem 2.5]{Brochard1} or \cite[1.4]{SGA4Tome3} for the precise nature of these equivalences.  

The history of Fourier-Mukai style duality theorems for commutative group stacks can be traced back to \cite{TFib} where Arinkin credits Be\u{\i}linson for teaching him about duality for commutative group stacks. In \cite[4.3]{TFib} a Fourier-Mukai equivalence between certain gerby genus 1 surfaces is proved. In \cite[Appendix by Arinkin]{TFib} the question of a general Fourier-Mukai-style duality between the bounded derived categories of a commutative group stack and its dual was raised. After \cite{TFib} Fourier-Mukai duality theorems for commutative group stacks were applied in \cite{DModcharp} to study $\mathcal{D}$-modules. Finally in \cite[A.3]{GeoLang} Fourier-Muakai duality for Be\u{\i}linson 1-motives in the fpqc-topology was obtained. Our contribution is the following addition.

\begin{mainTheoremA}\label{mainthmB}\label{mainthm}
Let $\A\ra S$ be an abelian stack on a connected base scheme $S$ with $A$ its coarse moduli space. Suppose that $\A\ra A$ is a good moduli space. Let $\QC(\A)$ be the stable infinity category of quasi-coherent sheaves on $\A$ and $\Dbul(\A)$ the unbounded derived category of quasi-coherent sheaves on $\A$ and $\Db(\A)$ the bounded derived category of sheaves on $\A$. Then we have that
\begin{align}
	\QC(\A)&\cong\QC(\dua(\A)).\label{eq:mainQC}\\
	\Dbul(\A)&\cong\Dbul(\dua(\A)).\label{eq:mainD}\\
	\Db(\A)&\cong\Db(\dua(\A))\label{eq:mainDb}.
	\end{align}	
\end{mainTheoremA}
Both of \ref{eq:mainD} and \ref{eq:mainDb} follow from \ref{eq:mainQC}, so $\QC(\A)\cong\QC(\dua(\A))$ is the fundamental result.\\
Our approach has the following steps.

\begin{enumerate}
	\item Prove that $\QC(\A)=\QC(\dua(\A))$.
	\item We conclude that $\Dbul(\A)=\Dbul(\dua(\A))$ by taking homotopy and that the equivalence is given by a Fourier-Mukai transform.
	\item We conclude that $\Db(\A)=\Db(\dua(\A))$ by verifying that this functor preserves bounded complexes. 
\end{enumerate}	
The first step is the most crucial. The structure morphism $\A\ra S$ allows us to consider $\QC(\A)$ as a sort of algebra over $\QC(S)$. We find that for an Abelian stack $\A\ra S$ that $\A$ is perfect, see \ref{thm:perfbro}. By \cite{bfn} the morphism $\A\ra S$ is perfect and $\QC(\A)$ is self-dual over $\QC(S)$. We then show that $\QC(\dua(\A))$ is dual to $\QC(\A)$ by constructing an explicit duality using a stacky Poincare sheaf on $\A\times \dua(\A)$. So, we have
\begin{equation}
	\xymatrix{\QC(\A)\ar[r]^{\cong}_{\textnormal{duality}} & \QC(\A)^\vee\ar[r]^{\cong}& \QC(\dua(\A))}
\end{equation}
This gives the desired equivalence; we then verify directly that it preserves bounded complexes. We also construct new examples of tame abelian stacks over non-closed fields for which our theorem applies and show that every tame-abelian stack over an algebraically closed field is trivial. One source of tame abelian stacks over a non-closed field $k$ are non-zero points in the group $A(k)/nA(k)$ where $A$ is an abelian variety over $k$. We consider $A(k)/nA(k)$ as a subgroup of $\Ext^1(\bZ/n\bZ,A)$ so that each point corresponds to an extension $0\ra A\ra G\ra \bZ/n\bZ \ra 0$. In such a situation, $\dua(G)$ is an abelian stack. 

In both algebraic geometry and topology, commutative group stacks and related notions have a number of important applications. In algebraic geometry, Behrend and Fantechi used commutative group stacks to define virtual fundamental classes, which are a foundational part of modern enumerative geometry, in particular, Gromov-Witten theory and Donaldson-Thomas theory. See \cite{MR1437495} for the construction. Commutative group stacks also arise in situations where duality is present. 
\begin{definition}{Definition 3.1 in \cite{Brochard1}}\label{def:dualGroup}
Let $\G$ be a commutative group stack defined over $S$. Define
\[\fD(\G):=\hom_{S-\textnormal{grp}}(\G,B\bG_m).\]
\end{definition}
This notion of duality generalizes both the classical Cartier duality and the biduality of abelian varieties. Deligne's 1-motives in \cite{MR498551}  provide further examples of commutative group stacks; moreover, Deligne's definition of the dual motive is the dual commutative group stack. Commutative group stacks have been used by Be\u{\i}linson and Bernstein in \cite{MR1237825} in relation to the Jantzen conjectures. In \cite{Brochard1} Brochard used commutative group stacks to generalize the albanese morphism in the context of algebraic stacks. 

Fix a base scheme $S$ and work in the fppf topology. Recall that every commutative group stack $\G$ arises as a quotient stack $\G=[G^0/G^{-1}]$ where $[G^0\ra G^{-1}]$ is a two-term complex of abelian sheaves on $S$. Associated to the two-term complex are its cohomology in degree $-1$ and $0$ which we denote $H^0(\G), H^{-1}(\G)$.  Geometrically $H^0(\G)$ is moduli $\emph{sheaf}$ of isomoprhism classes of objects in $\G$  defined by the sheafication of the pre-sheaf $T\ra \G(T)/\textnormal{isomorphism}$. Note that in essentially all cases we will consider $H^0(\G)$ will be representable by a commutative group scheme, and will often be a good moduli space in the sense of Alper for $\G$. The fact that $\A\ra H^0(\G)$ is a good moduli space will be a consequence of our assumption of tameness. On the other hand, $H^{-1}(\G)$ is the isomorphism group of the identity section of $\G$.
\begin{definition}[Brochard: See Proposition 2.14]\label{def:AbelianStack}
An \emph{abelian stack} is a commutative group stack $\A\ra S$ such that satisfies the following equivalent conditions.
\begin{enumerate}
    \item $\A\ra S$ is an algebraic, proper, flat, finitely presented with connected and reduced geometric fibres. Its inertia stack is finite, flat, and of finite presentation over $\G$.
    \item $H^0(\A)\ra S$ is an abelian scheme and $H^{-1}(\A)\ra S$ is a finite, flat, and finitely presented commutative group scheme.
\end{enumerate}
\end{definition}
\begin{example}
Let $A\ra \spec k$ be an abelian variety on a field $k$ equipped with a group homomorphism $\delta\colon F\ra A$ where $F$ is a finite commutative group scheme over $k$. Then $F$ acts on $A$ via $\delta$ and the quotient stack \begin{equation}
\A:=[A/F]\ra \spec k \end{equation} is an example of an abelian stack. 
\end{example} 
To compute the dual of an abelian stack, Brochard introduced the following notion. 

\begin{definition}[Duabelian groups: See definition 2.17 in \cite{Brochard1}]\label{def:duabelianGr}
Let $G$ be a sheaf of abelian groups on $S$. We say that $G$ is a \emph{duabelian} group if there is a finite flat and finitely presented commutative group scheme $F\ra S$ and an abelian scheme $A\ra S$ such that $G$ is an extension of $F$ by $A$. In other words, there is an exact sequence of commutative groups 
\begin{equation*}
0\ra A\ra G\ra F\ra 0.
\end{equation*}
In particular, a duabelian group is a commutative group scheme. 
\end{definition}
Brochard proved that the duabelian groups and abelian stacks are \emph{dual} to one another in the following sense.
\begin{theorem}[Abelian stacks and duabelian groups are dual: See Theorem 3.17 in \cite{Brochard1}]\label{thm:AbduaDuality}
Let $\A\ra S$ be an abelian stack and $G\ra S$ a duabelian group. Then $\dua(\A)$ is a duabelian group and $\dua(G)$ is an abelian stack. In particular, if $\A$ is an abelian stack with $H^0(\A)=M$ and $H^{-1}(\G)=F$ then we have a short exact sequence
\begin{equation}
0\ra M^t\ra \dua(\G)\ra F^D\ra 0     
\end{equation}
where $M^t\ra S$ is the dual abelian scheme and $F^D=\hom_{\textnormal{hom-grp}}(F,\bG_m)$ is the Cartier dual of $F$.
\end{theorem}

We now explain the relationship between our work and that in \cite{GeoLang} which has two main forms.

\begin{enumerate}
    \item In \cite{GeoLang} they work with fpqc stacks while we work with the fppf topology.
    \item There are Abelian stacks that are not obviously Be\u{\i}linson 1-motives.
\end{enumerate}
Working with the fppf topology has the benefit of avoiding issues with sheafication and stackfication in the fpqc topology, which need not exist without introducing extra set-theoretic complications. See \cite[Page 30, Tag 020K]{stacks-project} for a related discussion and \cite[Theorem 5.5]{Wat75} for an example of a fpqc presheaf with no sheafication. As sheafication and stackification play an important role in the theory of commutative group stacks, it seems worthwhile to prove statements without resorting to the fpqc-topology. For example, in Deligne's correspondence between two term complexes $[d\colon G^{-1}\ra G^0]$ one first forms the quotient pre-stack $[G^0/G^{-1}]^{\textnormal{pre}}$ and then takes $[G^0/G^{-1}]$ to be the stackfication of the pre-stack whose objects are $G^0(U)$ and morphisms between $x,y\in G^0(U)$ those $f\in G^{-1}(U)$ with $df=y-x$. On the other hand, starting with a commutative group stack $\G$ one uses sheafification to define $H^0(\G)$ as the sheaf associated with the presheaf $U\mapsto \G(U)/\textnormal{isomorphism}$.  

\begin{definition}[Be\u{\i}linson 1-motives. See A.4 in \cite{GeoLang}]\label{def:Bel1Motive}
A Be\u{\i}linson 1-motive over $S$ is an fpqc-commutative group stack $\G\ra S$ with a filtration
\[0\subseteq \G_0\subseteq \G_1\subseteq \G_2=\G\]
such that 
\begin{enumerate}
    \item $\G_2/G_1=\Gamma$ where $\Gamma$ is locally a constant group on $S$.
    \item $\G_1/\G_0$ is an abelian scheme over $S$.
    \item $\G_0=BG$ where $G=(\Gamma^\prime)^D=\hom(\Gamma^\prime,\bG_m)$ for some finitely generated abelian constant group $\Gamma^\prime$ on $S$.
\end{enumerate}
Here $\G_i\subseteq \G_{i+1}$ means that $\G_i\ra \G_{i+1}$ is a homomorphism of commutative group stacks which is a fully faithful embedding. 
    
\end{definition}

If $\A\ra S$ is an abelian stack with $H^{-1}(\A)=F$ where $F$ is a finite diagonalizable group, then we have a short exact sequence 
\[0\ra BF\ra \A\ra H^0(\A)\ra 0\]
by \cite[Example 2.13]{Brochard1}. Therefore, when $H^{-1}(\A)$ diagonalizable $\A
\ra S$ is a Be\u{\i}linson 1-motive as a stack in the fpqc topology and therefore the derived bounded categories of quasi-coherent sheaves on $\A$ and $\dua(\A)$ are equivalent by the Fourier-Mukai transform (on the derived bounded category in the fpqc topology) determined by the Poincare bundle by \cite[Theorem A.4.6]{GeoLang}. However, not every abelian stack is a Be\u{\i}linson 1-motive in this way. See examples \ref{ex:tamexample2} and \ref{ex:tameex3} for constructions of non-trivial commutative group stacks with $H^{-1}(\A)$ a constant group scheme. 

\end{section}

\begin{section}{Preliminaries}

\begin{subsection}{Basics of Commutative Group stacks}\label{sec:CGSBasics}

Here, we review the notions of a commutative group stack. The main references are \cite{Brochard1} and \cite[XVIII 1.4]{SGA4Tome3}. We review the basic definitions for the convenience of the reader.	
	
\begin{definition}[Definition of Picard category]\label{def:PicardCategory}
A Picard category is a groupoid $G$ equipped with the following data.
\begin{enumerate}
	\item A functor \begin{equation*}
	+\colon G\times G\ra G
	\end{equation*}
\item A natural isomorphism of functors  \begin{equation*}
\lambda\colon +\circ (+,\textnormal{id}_G)\ra +\circ (\textnormal{id}_G,+) 
\end{equation*}
where $+\circ (+,\textnormal{id}_G)\colon G\times G\times G\ra G\times G\times G$ is the functor defined by $(x,y,z)\mapsto (x+y)+z$, and $+\circ (\textnormal{id}_G,+) $ is defined analogously.
\item A natural isomorphism of functors \begin{equation*}+\ra +\circ s
\end{equation*}
where $s\colon G\times G\ra G\times G$ is the natural switching functor defined by $s(x,y)=(y,x)$.
\end{enumerate}
These data are required to satisfy the following laws.

\begin{itemize}
	\item The hexagon and pentagon axioms. See \cite[XVIII 1.4.1]{SGA4Tome3} for the details. 
	\item For any objects $x,y$ of $G$ we have  \emph{equalities}
	\begin{align*}
		\tau_{y,x}\circ \tau_{x,y}&=\textnormal{id}_{x+y}\\
		\tau_{x,x}&=\textnormal{id}_{x+x}
	\end{align*}
\item For any object $x\in G$ the translation functor $t_y\colon G\ra G$ given by $t_x(y)=x+y$ is an equivalence of categories.
\end{itemize}
\end{definition}
We may now define a commutative group stack. Roughly speaking, this is stack $\G$ over some base scheme which is \emph{fibered in Picard categories}.

 \begin{definition}[Definition of a commutative group stack]\label{def:CGS}
Let $S$ be a base scheme. A commutative group stack over $S$ is a stack $\G\ra S$ equipped with the following data.

\begin{enumerate}
	\item A morphism of stacks 
	\begin{equation*}
		+\colon \G\times_S \G\ra \G
	\end{equation*}
\item two isomorphisms $\lambda$ and $\tau$ defined analogously to $\lambda$ and $\tau$ in definition \ref{def:PicardCategory} such that for any test scheme $T\ra S$ we have that $\G(U)$ with the induced restrictions of $+\vert_U,\lambda\vert_U,\tau\vert_U$ is a Picard category. 
\end{enumerate}
\end{definition}
To define the dual commutative group stack, we require the notion of a homomorphism. 
\begin{definition}[The definition of a homomorphism of commutative group stacks]
Let $\G_1,\G_2$ be commutative group stacks over a base scheme $S$. A homomorphism $f\colon\G_1\ra \G_2$ is a morphism of stacks along with a choice of a 2-isomorphism \begin{equation*}
	\alpha_f\colon f\circ +_{\G_1}\Rightarrow +_{\G_2}\circ (f,f)
\end{equation*}

In addition, the 2-morphism $\alpha_f$ is required to satisfy some coherence conditions. See \cite[Definition 2.4]{Brochard1} and/or \cite[XVIII 1.4.6]{SGA4Tome3}. We let $\hom_{\textnormal{grp}}(\G_1,\G_2)$ be the stack of all homomorphisms of commutative group stacks.
\end{definition}

\begin{remark} \label{r:deligne}
	There is a 2-category of commutative group stack. The associated 1-category of commutative group stacks has objects commutative commutative group stacks 
	and morphisms isomorphism classes of morphisms in the 2-category of commutative group stacks. The associated 1-category has a simple description in terms of 
	two term complexes of sheaves of abelian groups concentrated in degrees -1 and 0. Given such a complex, $I^{-1}\rightarrow I^0$, the quotient stack 
	$[I^0/I^{-1}]$ is a commutative group stack. This construction induces and equivalence between the associated 1-category of commutative group stacks and 
	the derived category $D^{[-1,0]}(S)$ of abelian sheaves on $S$. See \cite[XVII 1.4.14 and 1.4.17]{SGA4Tome3} for more details. 
\end{remark}

\begin{example}\label{ex:classStack}
Let $G$ be a commutative group scheme over $S$. The addition on $G$ induces a canonical commutative group stack structure on the classifying stack $BG\ra S$. 
Alternatively, this follows from the prior remark by considering the complex concentrated in degree -1 with value $G$.
\end{example}

Following \cite{Brochard1} we denote by $\Hom_{cgs}(G,H)$ the groupoid of homomorphisms in the 2-category of commutative group stacks between the
commutative group stacks $G$ and $H$.

\begin{definition}[The definition of the dual commutative group stack: See definition 3.1 in \cite{Brochard1}]
Let $\G$ be a commutative group stack over $S$. The dual commutative group stack is defined to be 
\begin{equation}
	\dua(\G):=\Hom_{\textnormal{grp}}(\G,B\bG_m).
\end{equation}	
It assigns to the $S$ scheme $U$ the groupoid $\Hom_{cgs}(\G\times_S U, \bG_m\times_S U)$.
\end{definition}

\begin{remark}\label{rem:idrem}
Commutative group stacks are always equipped with a pair $(e,\epsilon)$ $e\colon S\ra \G$ is a section of the natural structure map $\pi\colon \G\ra S$ and $\epsilon\colon e+e\rightarrow e$ an isomorphism. This data is unique up to unique isomorphism. There is also a natural inverse involution $-\colon G\ra G$ that is unique up to unique isomorphism. See the references in \cite[Remark 2.3]{Brochard1} for the details. 
\end{remark}

\begin{definition}[The coarse moduli sheaf and automorphism group of the neutral section.]\label{def:H0}
Let $\G\ra S$ be a commutative group stack over $S$ with neutral section $e_\G\colon S\ra \G$.

\begin{itemize}
	\item Let $H^0(\G)$ be the fppf sheafication of the fppf-pre sheaf defined by the rule
	\begin{equation}
		U\mapsto \G(U)/\textnormal{isomorphism}.
	\end{equation}
\item Define $H^{-1}(\G)$ as the following fiber product:
\begin{equation}
	\xymatrix{H^{-1}(\G)\ar[r]\ar[d]&I_\G\ar[r]\ar[d] & \G\ar[d]^\triangle\\S\ar[r]_{e_\G}& \G\ar[r]_{\triangle} & \G\times_S \G}
\end{equation}
\end{itemize}
Recall from \ref{r:deligne} that commutative group stacks correspond to two term complexes $[G^{-1}\ra G^0]$. 
The notation $H^0(\G)$ stems from the fact that in this language $H^0(\G)$ is the cohomology in degree $0$ and $H^{-1}(\G)$ is the cohomology in degree $-1$.\end{definition}

\begin{proposition}
	Let $\G$ be a commutative group stack over $S$. The following are equivalent
	\begin{enumerate}
		\item $\G$ is algebraic, proper, flat and of finite presentation with conencted and reduced geometric fibers. Further $I_{\G}$, the inertia 
		the inertia stack of $\G$, is finite, flat and of finite presentation over $\G$.
		\item $H^{-1}(\G)$ is a finite, flat and finitely presented group scheme over $S$ and $H^0(\G)$ is an abelian scheme.
	\end{enumerate}
\end{proposition}

\begin{proof}
	The is \cite[Proposition 2.14]{Brochard1}.
\end{proof}

When these equivalent conditions are satisfied we call $\G$ an abelian stack.

\begin{theorem}[The dual of an abelian stack]\label{thm:dualAbStack}
Let $\A\ra S$ be an abelian stack. Let $e_\A\colon S\ra \A$ be the neutral section of $\A$. Define $F$ by the following fiber diagram.
	
Then $\dua(\A)$ is a group scheme over $S$ and fits into an exact sequence
\begin{equation}
	0\rightarrow A^t \rightarrow \dua(\A) \rightarrow F^D\rightarrow 0
\end{equation}
where $F^D=\hom_{
\textnormal{grp}}(F,\bG_m)$ is the Cartier dual of $F$ and $A^t$ is the dual abelian scheme of $A$. Furthermore, $H^{-1}(\G)=F$.
\end{theorem}
\begin{proof}
	This is shown in the proof of \cite[Theorem 3.17]{Brochard1}.
\end{proof}

By the definition of an abelian stack, we find that the group scheme $F\ra S$ defined above is finite and flat, because the inertia of $\A\ra S$ is finite and flat by the definition of an abelian stack. By Cartier duality we find that $F^D$ is finite and flat. In other words, if $\A$ is an abelian stack, then $\dua(\A)\ra S$ is an extension of a finite flat group scheme by an abelian scheme. There is a converse to theorem \ref{thm:dualAbStack}.

\begin{definition}[Duabelian groups: See definition 2.17 in \cite{Brochard1}]\label{def:duabelianGr}
Let $G\ra S$ be a commutative group scheme over $S$. We say that $G$ is a \emph{duabelian} group if there is a finite flat and finitely presented commutative group scheme $F\ra S$ and an abelian scheme $A\ra S$ such that $G$ is an extension of $F$ by $A$. That is there is an exact sequence of commutative group schemes
\begin{equation*}
0\ra A\ra G\ra F\ra 0.
\end{equation*}
\end{definition}
Brochard proved that the duabelian groups and abelian stacks are \emph{dual} to one another in the following sense.
\begin{theorem}[Abelian stacks and duabelian groups are dual: See theorem 3.17 in \cite{Brochard1}]\label{thm:AbduaDuality}
Let $\A\ra S$ be an abelian stack and $G\ra S$ be a duabelian group. Then $\dua(\A)$ is a duabelian group and $\dua(G)$ is an abelian stack. 
\end{theorem}

\end{subsection}

\end{section} 

\begin{section}{Infinity categories}

\begin{subsection}{The stable $\infty$-category of quasi-coherent sheaves}\label{subsec:QC}

We refer the reader to \cite[Chapter 1]{HAG} for an introduction to stable infinity categories. In particular, given an ordinary 
ring $R$ one can associate to it a stable infinity category denoted $\QC(\spec R)$ whose homotopy category is the ordinary
unbounded derived category of $R$-modules, see \cite[1.2]{HAG}. For an algebraic stack $X$ we define, following \cite{bfn},
\[
	\QC(X) = \lim_{\spec R\in {\rm Aff}/X} \QC(\spec R).
	\]
The limit is taken in the $\infty$-category of $\infty$-categories. This is a stable $\infty$ category with homotopy category the usual
unbounded derived category of complexes of sheaves with quasi-coherent cohomology. 
We will refer to the vertices of $\QC(X)$ as complexes of ${\mathcal O}_X$-modules. 
\end{subsection}

\begin{subsection}{Perfect stacks}
We will use the notion of a perfect stack introduced in \cite[Section 3]{bfn}. Recall that a complex of $R$-modules $M\in \QC(\spec R)$ is called
perfect if it is perfect, a perfect complex in the usual sense. For a stack $X$, a complex $M\in\QC(X)$ is said to be perfect if for every morphism
$f:\spec R\rightarrow X$ the restriction $f^*M$ is perfect. 

We denote by ${\rm Perf}(X)$ the full $\infty$-subcategory of $\QC(X)$ generated by perfect complexes. We say $X$ is perfect, see \cite[3.2]{bfn}, 
if $X$ has affine diagonal and 
\[
	\QC(X)\cong {\rm Ind}{\rm Perf}(X).\]
	We refer the reader to \cite[Sec. 3]{bfn} for the definitions of perfect objects, compact, and dualizable objects.

We have the following characterization of perfect stacks.

	\begin{proposition}\label{p:bz}
		Suppose $X$ is a stack with an affine diagonal. Then the following are equivalent
		\begin{enumerate}
			\item $X$ is perfect 
			\item $\QC(X)$ is compactly generated and its compact and dualizable objects coincide.
		\end{enumerate}
	\end{proposition}
	\begin{proof}
		This is \cite[Prop. 3.9]{bfn}.
	\end{proof}

A morphism of algebraic stacks $X\rightarrow Y$ is perfect if for all $\spec R\rightarrow Y$ the stack $X\times_Y \spec R$ is perfect. 

For our application, we will need to know that $\A\ra S$ is perfect.

\begin{lemma}
	Let $X$ be a quasi-separated algebraic stack. Then every compact object of $\QC(X)$ is perfect.
\end{lemma}
\begin{proof}
	This is \cite[Lemma 4.4]{hr}.
\end{proof}

\begin{proposition}\label{p:concentrated}
	Let $X$ be a quasi-compact, quasi-separated algebraic stack. The 
	following are equivalent:
	\begin{enumerate}
		\item Every perfect complex is compact.
		\item $X$ has finite cohomological dimension, that is, there is an integer 
		$N$ so that $H^d(X,F)$ vanishes for $d>N$ and every quasi-coherent sheave $F$.
	\end{enumerate}
\end{proposition}

\begin{proof}
	See \cite[Remark 4.6]{hr}.
\end{proof}

\begin{theorem}\label{thm:percompactgeneration}
	Let $X$ be a quasi-compact algebraic stack with quasi-finite, quasi-separated diagonal. Then $\QC(X)$ is generated by perfect complexes. 
\end{theorem}

\begin{proof}
	This is \cite[Theorem A]{hr}.
\end{proof}

\begin{theorem}\label{thm:perfbro}
	Let $\G$ be a proper, flat, and finitely presented commutative group stack over a quasi-compact and separated scheme $S$. Suppose that $\G$ has finite, flat, and finitely presented inertia over $S$. Additionally, assume that the moduli map $f:\G\rightarrow H^{0}(\G)$ is a good moduli space morphism and that $H^0(\G)\ra S$ has finite cohomological dimension. Then $\G$ is a perfect stack and $\G\ra S$ is a perfect morphism.
\end{theorem}

\begin{proof}
	$\G\ra \spec S$ is proper so it is separated and so quasi-separated. By Theorem \ref{thm:percompactgeneration} we see that $\QC(\G)$ is generated by a perfect complex. Since the moduli map is good, it is cohomologically affine, which means that all higher direct images of quasi-coherent sheaves vanish. The Leray spectral sequence then gives $H^p(\G,F)=H^p(H^0(\G),f_* F)$ for any quasi-coherent sheaf $F$ on $\G$. As we have assumed that $H^0(\G)\ra S$ has finite cohomological dimension, and because $f_*F$ is quasi-coherent we obtain that $\G\ra S$ has finite cohomological dimension. So by proposition \ref{p:concentrated} all perfect complexes are compact. Since every compact object is perfect, and perfect objects already coincide with dualizable objects we may conclude that $\G$ is perfect by applying proposition \ref{p:bz}. Finally, since $S$ is quasi-compact and quasi-separated we have $S$ perfect. So $\G\ra S$ is a morphism between perfect stacks. Therefore, by \cite[Corollary 3.23.]{bfn} we find that $\G\ra S$ is perfect.
\end{proof}

\begin{corollary}[Tame abelian stacks are perfect]\label{prop:abStPerfect}
Let $\A\ra S$ be a tame abelian stack over a Noetherian base scheme. Then $\A$ and $S$ are perfect and $\A\ra S$ is a perfect morphism. 
\end{corollary}
\begin{proof}
Since $\A\ra S$ is an abelian stack, we see that $\A\ra S$ is flat, proper, and finitely presented, with finite and flat inertia by \cite[Proposition 2.14]{Brochard1}. Since we assumed that $\A\ra S$ is also tame, we see that $\A\ra H^0(\A)$ is a good moduli space. On the other hand, since $\A\ra S$ is an abelian stack, $A=H^0(\G)$ is an abelian scheme over $S$. As $S$ is Noetherian we find that the higher cohomology of quasi-coherent sheaves vanish on $A$. Therefore $A\ra S$ has finite cohomological dimension and therefore $\A\ra S$ satisfies the assumptions of Theorem \ref{thm:perfbro} and we conclude that $\A\ra S$ is a perfect morphism between perfect stacks.  
\end{proof}

\end{subsection}

\end{section}

\begin{section}{Examples of tame perfect stacks}\label{sec:examples}
In this section we construct non-trivial examples of tame abelian stacks. 
\begin{proposition}
Let $\A\ra S$ be an abelian stack. Then $\pi \colon \A\ra H^0(\A)$ is a good moduli space if and only if the geometric fibres of $H^{-1}(\A)\ra S$ are diagonlizable.     
\end{proposition}
\begin{proof}
If $\pi\colon\A\ra H^0(\G)$ is a good moduli space then $\A$ is a tame stack by \cite[Remark 4.3]{GoodModuli1}. By \cite[Corollary 3.5]{TameStacks} this is equivalent to the geometric fibers of $\A\ra S$ are tame stacks. By \cite[Theorem 3.2]{TameStacks} this is equivalent to the automorphism groups of points of $\A\times_S \spec k$ are linearly reductive for any geometric point $\spec k\ra S$. Because $\A\times_S \spec k$ is a group stack over a field all the automorphism groups are isomorphic to $H^{-1}(\A\times_S \spec k)=H^{-1}(\A)\times_S \spec k$. By \cite[Lemma 2.5]{Linred} this is equivalent to $H^{-1}(\A)\times_S \spec k$ being diagonalizable as claimed. 
\end{proof}

By \ref{thm:AbduaDuality} and \ref{thm:dualAbStack} constructing abelian stacks is equivalent to constructinga duabelian group
\[0\ra A\ra G\ra F\ra 0\]
where $F$ is is a finite flat commutative group scheme which has linearly reductive geometric fibers and $A$ is an abelian scheme. An abelian stack is obtained by considering $\dua(G)$. Therefore, we seek to construct non-trivial extensions of this form. 
\begin{proposition}
Let $S=\spec k$ where $k$ is an algebraically closed field. Then every tame abelian stack on $S$ is trivial.
\end{proposition}
\begin{proof}
Let $\A$ be a tame abelian stack over $k$. Then $\A=\dua(G)$ where $G$ is a duabelian group. As $G$ is duabelian we have an exact sequence
\[0\ra A\ra G\ra F\ra 0\]
where $A$ is an abelian variety over $k$ and $F$ is a finite commutative group scheme over $k$ by \ref{def:duabelianGr}. Furthermore, we have $F^D=H^{-1}(\A)$ by the proof of \cite[Theorem 3.17]{Brochard1}. As $\A$ is tame we have $H^{-1}(\A)$ must be linearly reductive. Therefore, we see that $F^D$ is diagonalizable. We conclude $F^D=\prod_{i}\mu_{n_i}$ by \cite[2.5]{Linred}. Taking the dual gives thats that $F=\prod_{i=1}\bZ/n_i\bZ$. We now compute 
\[\Ext^1(F,A)=\Ext^1(\prod_i\bZ/n_i\bZ,A)=\prod_i\Ext^1(\bZ/n_i\bZ,A).\]
By \cite[Section 10]{Breen} $\Ext^1(\bZ/n_i\bZ,A)=0$ when $k=\bar{k}$ and therefore we have $\Ext^1(F,A)=0$. We conclude that $G=A\times F$. Taking the dual gives $\dua(A\times F^D)=A^t\times B(F^D)$ and so $\A$ is trivial as claimed. 

\end{proof}
The tameness assumption is necessary.
\begin{example}
Let $S=\spec k$ where $k$ is algebraically closed of characteristic $p>0$. Let $A\ra S$ be an abelian variety. Following Oort in \cite[II.12-3]{OortBook} set $\tau(A)=\dim_k\hom(\alpha_p,A)$. Suppose that $\tau(A)>0$. The table in \cite[II.14-2]{OortBook} gives that $\Ext^1(\alpha_p,A)\cong k^{\tau(A)}$. Therefore, there are nontrivial extensions \[0\ra A\ra G \ra \alpha_p\ra 0\]
We conclude that there are nontrivial abelian stacks of the form $\dua(G)$. Note that Oort does not work in the category of fppf-sheaves, but by the remarks in the introduction in \cite{Breen} the calculation of $\Ext^1$ remains valid there over a field. So there are non-tame group stacks over algebraically closed fields of characteristic $p>0$.

\end{example}

\begin{proposition}\label{prop:constructionprop}
Let $A\ra S$ be an abelian scheme over a smooth connected Noetherian base. Fix $n\geq 2$ and assume that the multiplication by $n$ on sections \[n:A(S)\ra A(S)\]
is not surjective. Then $\Ext^1(\bZ/n\bZ,A)\neq 0$.
    
\end{proposition}
\begin{proof}
Take the long exact sequence associated to
\[0\ra \bZ_S\ra \bZ_S\ra (\bZ/n\bZ)_S\ra 0\] and $\hom(-,A)$ associated to this sequence. We obtain
\[\hom(\bZ_S,A)\ra \hom(\bZ_S,A)\ra \Ext^1((\bZ/n\bZ)_S,A)\]
As $\hom(\bZ_S,A)=A(S)$ and the morphism between them is multiplication by $n$ on the sections. By assumption the map is not surjective, therefore we obtain non-trivial elements in $\Ext^1((\bZ/n\bZ)_S,A)$.

\end{proof}

\begin{example}\label{ex:tameexample1}
Fix a field $k$ non-algebraically closed and an abelian variety over $k$ such that $n:A(k)\ra A(k)$ is not-surjective. For example, if $k=\bQ$ take $A$ to be an elliptic curve with $A[2](\bQ)=\bZ/2\bZ$ and $n=2$. Alternatively take $k=\mathbb{F}_p$ and $n$ not divisible by $p$ with an $n$-torsion point. Then by proposition\ref{prop:constructionprop} we have a non-trivial extension
\[0\ra A\ra G\ra (\bZ/n\bZ)_k\ra 0\]
Dualizing gives an abelian stack $\dua(G)$ with $H^{-1}(\dua(G))=(\bZ/n\bZ)^D=\mu_n$.
\end{example}

\begin{example}\label{ex:tamexample2}
Work on a base scheme $S$ and choose an abelian scheme $A\ra S$ such that $\Ext^1(\mu_2,A)\neq 0$ as in example \ref{ex:tameexample1}. Fix an odd prime $p$ so that $\mu_{2p}=\mu_2\times \mu_p$. Therefore, $\Ext^1(\mu_{2p},A)=\Ext^1(\mu_p,A)\times \times\Ext^1(\mu_2,A) $. As $\Ext^1(\mu_2,A)\neq 0$, we have $\Ext^1(\mu_{2p},A)\neq 0$. Therefore, we have non-trivial extensions \[0\ra A\ra G\ra \mu_{2p}\ra 0\]
Taking the dual gives an abelian stack $\A=\dua(G)$ with $H^{-1}(\A)=(\mu_{2p})^D=\bZ/2p\bZ$.
\end{example}

\begin{example}\label{ex:tameex3}
Consider an abelian variety over a field $k$ with no $k$-rational $2$-torsion point and $2\colon A(k)\ra A(k)$ not surjective. Take the long exact sequence associated 
\[0\ra \mu_{2}\ra \mu_4\ra \mu_2\ra 0\]
 and $\hom(-,A)$ to obtain $0\ra \hom(\mu_2,A)\ra \hom(\mu_4,A)\ra \hom(\mu_2,A)\ra \Ext^1(\mu_2,A)\ra \Ext^1(\mu_4,A)$. By our assumptions, we have $\hom(\mu_2,A)=0$ and $\Ext^1(\mu_2,A)\neq 0$. Therefore, $\Ext^1(\mu_4,\neq 0)$. Now assume that we have shown that $\Ext^1(\mu_{2^{n-1}},A)$ is non-zero. We have a short exact sequence
     \[0\ra \mu_{2}\ra \mu_{2^n}\ra \mu_{2^{n-1}}\ra 0\]
 Taking the long exact sequence as before gives \[\hom(\mu_2,A)\ra \Ext^1(\mu_{2^{n-1}},A)\ra \Ext^1(\mu_{2^n},A)\]
 which shows that $\Ext^1(\mu_{2^n},A)\neq 0$. Therefore, we obtain examples of abelian stacks $\A$ with $H^{-1}(\A)=\bZ/2^n\bZ$.
\end{example}

\end{section}

\begin{section}{The Poincare Bundle}

In this section, we discuss the Poincare bundle on an abelian stack. Our approach will be to show that the study of the Poincare bundle on an abelian stack can be reduced to the study of the Poincare bundle on the coarse moduli space, where we may apply the results of Mukai on relative Fourier-Mukai transforms.

\begin{definition}[The Poincare bundle]\label{def:PoBundle}
Let $G$ be a commutative group stack that is flat, proper, and finitely presented over $S$. Assume in addition that the inertia stack of $G$ is finite, flat, and finitely presented over $G$.
Then by \cite[Theorem 3.14]{Brochard1} $\dua(G)=\Hom_{\textnormal{grp}}(G,B\bG_m)$ is an algebraic stack. By \cite[Definition 4.1]{Brochard1} 
there is a canonical evaluation morphism

\begin{align}
\eval_G\colon G\times_S\dua(G)&\ra B\bG_m\\
(g,\varphi)&\mapsto \varphi(g)
\end{align}
We define the Poincare bundle on $G\times \dua(G)$ by
\begin{equation}
\P_G=\eval_G^*(\L_{B\bG_m})
\end{equation}
where $\L_{B\bG_m}$ is the universal line bundle on $B\bG_m$. 
\end{definition}
The objects of $\dua(G)(T)$ can be thought of as pairs $(L,\alpha)$ where $L$ is a line bundle on $G\times_S T$ and 
$\alpha\colon m_G^*L\ra p_1^*L\otimes p_2^* L$
 an isomorphism satisfying some coherence conditions; see the discussion preceding \cite[Theorem 3.14]{Brochard1}.
Here $m_\G\colon \G\times_S \G\ra \G$ is the multiplication and $p_i\colon \G\times_S \G\ra \G$ the projections. 

Let $M\ra S$ be an abelian scheme, and $M^t\ra S$ its dual. The normalized Poincare bundle on $M\times_S M^t$ is a line bundle $P_M$ such that if $L$ is a line bundle on $M\times_S T$ that is numerically trivial, then $L=(1\times h)^*P_M\otimes f_T^* N$ where $h\colon T\ra M^t$ is the morphism giving $L$ and $N$ 
is some line bundle on $T$. Normalized refers to the fact that $P\vert_{0_M\times M^t}=\O_{M^t}$ and $P\vert_{M\times 0_{M^t}}=\O_A$. For an abelian scheme, a normalized Poincare bundle always exists as the Picard scheme exists, as the morphism is cohomologically flat and there is also a section (see Chapter 8 of\cite{BLR} for details.)
  We can immediately prove that this notion generalizes the classical notion.

\begin{proposition}\label{prop:AbUniPoincare}
Let $M$ be an abelian scheme over a base $S$ and $M^t$ the dual abelian scheme. Then $\P_M$ is the normalized Poincare bundle. 
\end{proposition}
\begin{proof}
Note that $\P_M\vert_{0_M\times_S \dua(M)}$ corresponds to the morphism to the evaluation at $0_M$ homomorphism 
$M^t\ra B\bG_m$ which is the trivial homomorphism. Therefore, $\P_M\vert_{0_M\times_S \dua(M)}$ is trivial. 
On the other hand $\P_M\vert_{M\times 0_{M^t}}$ corresponds to evaluating the trivial homomorphism, which is trivial.

 On the other hand, $P_M$ satisfies the universal property of the Poincare bundle. Given a section $\sigma\colon S\ra M$ 
 we have that the restriction to the fiber $\P_M\vert_{\sigma\times M^t}$ is given by the group homomorphism $\eval(\sigma,\bullet)\colon M^t\ra B_S\bG_m$. However, as the evaluation morphism gives a natural isomorphism between $(M^t)^t\cong M$ we see that the family of lines bundles on $M^t$ given by restricting $\P_M$ to the fibers of $M\times_S M^t\ra M$ determines $M$.
\end{proof}

We now prove a series of lemmas that will allow us to work with the Poincare bundle on tame abelian stacks. Abelian stacks are gerbes over their coarse moduli space as they have finite and flat inertia. We will require the following basic results.

Throughout $A\rightarrow S$ will be an abelian stack with coarse moduli map $f:A\rightarrow M$ so that $M$ is an abelian scheme. Recall that Theorem \ref{thm:dualAbStack}
gives a canonical closed immersion $i:M^t\hookrightarrow \dua(A)$. 

\begin{proposition}\label{prop:transposeCommute}
	In the above setting, we have a commutative diagram:

	\begin{center}	
		\begin{tikzcd}
			A \times_S M^t \arrow[r, "1\times i"] \arrow[d, "f\times 1"] & A  \arrow[d, "ev"] \times_S \dua(A) \\ 
			M \times_S M^t \arrow[r, "ev"] & B\bG_m. 
		\end{tikzcd}
	\end{center}
\end{proposition}

\begin{proof}
	A $T$-point of $A\times_S M^t$ amounts to a pair $(\phi, L)$ where $\phi$ is a $T$-point of $A$ and $L$ is a line bundle on $M$. 
Composing this $T$-point with $(1\times i)$ amounts to the pair $(\phi, f^*L)$ and then evaluating produces a $T$-point of $B\bG_m$ which is
given by the line bundle $(\phi^*f^*L)$

Now, composing the same $T$-point with the morphism $(f\times 1)$ produces the pair $(f\circ \phi, L)$, and hence the evaluation morphism produces the same line bundle $\phi^*f^* L$.
\end{proof}
 
\begin{corollary} \label{cor:pushforward}
	In the above situation, we have $(f\times 1)_*(1\times i)^* \P_A=\P_M$
\end{corollary}
\begin{proof}
	The morphism $f$ is a good moduli map in the sense of \cite{GoodModuli1}. The result follows from \cite[Theorem 10.3]{GoodModuli1}.
\end{proof}

\begin{lemma}\label{lem:pshlem}
Let $S$ be a Noetherian scheme. Let $G\ra S$ be a finite-type, proper, and flat algebraic stack over $S$. Suppose that $\pi\colon G\ra M$ is a coarse space
map for $G$  so that $\pi$ is a good moduli space morphism and that $M$ is irreducible. Let $L$ be a line bundle on $G$. 
Suppose that there exists a field-valued point $x\colon \spec k \ra M$ with closed image such that $L$ has non-trivial inertial action at $x$. 
Then $\pi_* L=0$. 
\end{lemma}
\begin{proof}
 Since $\pi$ is cohomologically affine, we have $R^i\pi_*L=0$ for all $i>0$. 
 In particular, $\pi_* L$ is a vector bundle on $G$ by \cite[Theorem A]{Hall14}.
As $M$ is connected, the vector bundle $\pi_* L$ has constant rank on $M$, so it suffices to show that it is $0$. But this follows by pulling back to the point $x$ and applying \cite[Theorem A]{Hall14} again. 
 \end{proof}

\begin{lemma}
	Let $G\rightarrow S$ be as in the previous lemma, with a good moduli map $f:G\rightarrow M$. Let $L$ be an invertible sheaf on $G$. 
	If $f_*(L)\ne 0$ then $L=f^*(F)$ for some $F\in{\rm Pic}(M)$.
\end{lemma}

\begin{proof}   
Since $f_* L\neq 0$ we have by lemma \ref{lem:pshlem} that the inertia action at all points is trivial for $L$. 
Then by \cite[Theorem 10.3]{GoodModuli1} we have $L=f^*F$ for some line bundle on $M$.	
\end{proof}

\begin{lemma}\label{lem:inertiaaction}
	Let $A$ be an abelian stack over a base scheme $S$. Let $f\colon A\ra M$ be the coarse space mapping of $A$. Let $P_A$ be the Poincare bundle on $A\times \dua(A)$. Let $i\colon M^t\ra \dua(A)$ be the inclusion in \ref{thm:dualAbStack}.
	 Let $(a,x)\in A\times_S \dua(A)$ be a field-valued point. Then $P_A\vert_{A\times x}$ has trivial inertial action if and only if $x\in M^t$.
\end{lemma}
\begin{proof}
	From theorem \ref{thm:dualAbStack} we have an exact sequence
	\begin{equation}\label{eq:Brochardeq1}
		0\ra M^t\ra \dua(A)\ra F^D\ra 0
	\end{equation}
	Note that $F=H^{-1}(A)$. The point $x$ gives a homomorphism $x\colon A\ra B\bG_m$. Restricting to $H^{-1}(A)$ 
	gives $x\colon F\ra \bG_m$. On the other hand, the restriction to $P_{A}\vert_{A\times x}$ is a line bundle on $A$. The inertia action at the closed point $a$ is precisely given by $x\colon F\ra \bG_m$. If $x\in M^t$ this is the zero morphism by equation \ref{eq:Brochardeq1} and so we have the trivial inertia action. Conversely, the trivial inertia action is precisely the same as saying that $x$ lies in the kernel the morphism described by equation \ref{eq:Brochardeq1}, which is the desired result. 
\end{proof}

\begin{corollary}\label{cor:supportlem}
	Let $A$ be an abelian stack over $S$. Let $f\colon A\ra M$ be the coarse space mapping of $A$. Let $P_A$ be the Poincare bundle on $A\times \dua(A)$. Let $i\colon M^t\ra \dua(A)$ be the inclusion of the dual of an abelian stack.
	
	Consider the morphism
	\[f\times 1\colon A\times \dua(A)\ra M\times \dua(\A). \]
	Then $(f\times 1)_*P_\A$ is set-theoretically supported on $M\times M^t$ in $M\times \dua(A)$.
\end{corollary}
\begin{proof} 
	Consider a fiber over a closed point $(a,x)\in M\times M^t$. Now restrict to the fiber over $x$. If the inertia action of $P_A\vert_{A\times x}$ is non-trivial then $f_*(P_A\vert_{A\times x})=0$ by Lemma \ref{lem:pshlem}. 
	So, the support is contained in those points with a trivial inertia action. By lemma \ref{lem:inertiaaction} we see that this is precisely those points $x\in A^t$. 
\end{proof}

We next show that this in fact is a scheme-theoretic support. To do this,
it suffices to show that if $I$ is the ideal sheaf of $M^t\times M$ in $\dua(A)\times M$ then $I^2=I$. 

Let us record the following preliminary lemma.

\begin{lemma}{main one to check probably}\label{lem:FibralVanishing}
	Suppose that $\spec B\ra \spec R$ is a finite type morphism of Noetherian rings. Let $M$ be a finitely generated module on $B$. Suppose that for every prime ideal $\mathfrak{p}$ of $R$ we have $M\otimes_R k(\fp)=0$. Then $M=0$.  	
\end{lemma}
\begin{proof}
	Choose a maximal ideal $\fn$ of $B$ and set $\fp=R\cap \fn$. By assumption $M\otimes_R k(\fp)=0$. Note that as an $R_\fp$ module we have \[M\otimes_R (R_\fp/\fm_\fp)\cong (M\otimes_R R_\fp)\otimes_{R_\fp} (R_\fp/\fm_\fp)\]
	As $M\otimes_R R_\fp\cong M_\fp$ we have $M\otimes _R (R_\fp/\fm_\fp)\cong M_\fp\otimes_{R_\fp}(R_\fp/
	\fm_\fp)\cong M_\fp/\fm_\fp M_\fp=0$. In other words $\fm_\fp M_\fp=M_\fp$. Therefore, given $m\in M$ there are $p_1,\ldots, p_s\in \fp, s^\prime_1,\ldots s^\prime_t\in R- \fp$ and $m_1,\ldots, m_s\in M$ and $s_1,\ldots, s_t\in R-\fp$ with
	\[\frac{m}{1}=\sum_{i=1}^t \frac{p_i}{s_i^\prime}\frac{m_i}{s_i}.\]
	Write $s^{\prime\prime}=\prod_{i=1}^ts_is_i^\prime$ and $s^{\prime\prime}_i=\dfrac{s^{\prime\prime}}{s_is_i^\prime}$. Then there is some $s^{\prime\prime\prime}\in R- \fp$ with 
	\[s^{\prime\prime \prime}s^{\prime\prime}m=s^{\prime\prime\prime}\sum_{i=1}^ts^{\prime\prime}_i p_im_i\]
	Now consider $M_\fn/\fm_\fn M_\fn$. We have that $s^{\prime\prime\prime},s^{\prime\prime}s^{\prime\prime}_i, s^{\prime}_i, s_i$ are all non-zero in $B-\fn$ because $\fn\cap R=\fp$. So we have that
	\[\frac{m}{1}=\sum_{i=1}^t \frac{p_i}{s_i^\prime}\frac{m_i}{s_i}=\dfrac{\sum_{i=1}^tp_im_is_i^{\prime\prime}}{s^{\prime\prime}}\]
	in $M_\fn$. If $s\notin B- \fn$. So we have that
	\[s^{\prime\prime\prime}s^{\prime\prime}sm=ss^{\prime\prime\prime}\sum_{i=1}^tp_im_is_i^{\prime\prime}\] which gives $\frac{m}{s}=\sum_{i=1}^t\frac{p_im_is_i^{\prime\prime}}{ss^{\prime\prime}}=\sum_{i=1}^t\dfrac{m_ip_i}{ss_i^\prime s_i}$. As $\frac{p_i}{s_is_i^\prime s}\in \fm_\fn$ we have $M_\fn=\fm_\fn M_\fn$. As $M_\fn$ is finitely generated over $B_\fn$ and $\fm_\fn M_\fn=M_\fn$ we have by Nakayama's lemma that $M_\fn=0$ for an arbitrary maximal ideal $\fn$ of $B$. Consequently $M=0$.  
		
\end{proof}

\begin{proposition} 
	Let $S$ be a Noetherian scheme. Consider an abelian stack $A\ra S$ with $H^{-1}(A)=F$ and $M$ its moduli space. Assume that $F$ is tame. 
	Let $\I$ be the ideal of $M^t$ inside $\fD(A)$. Then $\I^2=\I$. 
\end{proposition}
\begin{proof}
	Consider the exact sequence $0\ra \I^2\ra \I\ra \C\ra 0$. We claim that $\C=0$ which can be checked affine locally. Therefore, we may assume $S=\spec R$.
By the previous lemma it suffices to show that $C\otimes_R k(p)=0$ for every prime in $R$. 
Now $I\otimes_R k(p)$ is the ideal sheaf of $M^t\times_S \spec k(p)$ inside $\fD(A)\times_S \spec(k(p))$ and dualization is compatible with base change. 
Hence we are reduced to the case where $R$ is a field. This is straightforward, as $M^t$ is a geometric component of $\fD(A)$
\end{proof}

\begin{corollary}\label{cor:keycor}
	The ideal sheaf of $M\times M^t$ inside $M\times \fD(A)$ is idempotent. 
\end{corollary}

\begin{corollary}
	Recall the cartesian diagram
	\begin{center}	
		\begin{tikzcd}
			A \times_S M^t \arrow[r, "1\times i"] \arrow[d, "f\times 1"] & A  \arrow[d, "f\times 1"] \times_S \dua(A) \\ 
			M \times_S M^t \arrow[r, "1\times i"] & A\times_S \fD(A). 
		\end{tikzcd}
	\end{center}
	We have $R(f\times 1)\P_A = (f\times 1)_*\P_A =  (1\times i)_*\P_M$
\end{corollary}

\begin{proof}
	The assertion that $R(f\times 1)\P_A = (f\times 1)_*\P_A$ follows from the fact that $f\times 1$ is cohomologically affine, as it is a good moduli map. The last assertion follows from Corollary \ref{cor:pushforward}.
\end{proof}

\begin{theorem} Continuing in the above situation. Let $\pi\colon \A\times_S \dua(\A)\ra \dua(\A)$. Let $\epsilon_{A^t}: S\ra \A^t$ be the identity section of $A^t$ 
	and $\epsilon_{\dua(\A)}$ the identity section of $\dua(\A)$. Then 
	$$R\pi_{\dua(\A)}\P_\A=\epsilon_{\dua(\A)*}\omega_{A^t/S}^{-1}[-g]$$ where $g$ is the relative dimension of $A\ra S$ and $\omega_{\A^t/S}$ the relative canonical bundle of $A^t/S$.
\end{theorem}

\begin{proof}
We have $(f\times \iden_{\dua(\A)})_* \P_\A\cong (\iden_A\times i)_*(\P_A)$ by Corollary \ref{cor:keycor}. 
 
  Using the computation on page 519 of \cite{relFM} or \cite[Lemma 1.2.5]{MR1159232} we have that $Rp_{A^t*}\P_A=\epsilon_{A^t*}\omega_{A^t/S}^{-1}[-g]$ 
  which gives that that $Rp_{\dua(\A)*}\P_\A\cong i_*\epsilon_{A^t/S*}\omega_{A^t/S}^{-1}[-g]\cong \epsilon_{\dua(\A)*}\omega_{A^t/S}^{-1}$ as claimed. 
\end{proof}

\end{section}

\begin{section}{Proof of duality}
We first verify that the Poincare bundle gives an object of $\QC(\A\times_S \dua(\A))$
\begin{proposition}
		The Poincare bundle gives an object of $\QC(\A\times \dua(\A))$.
	\end{proposition}
	
	\begin{proof}
Consider $\A\times_S \dua(\A)$ as a spectral Deligne-Mumford stack in the sense of Lurie. In 2.2.6 in spectral algebraic geometry by Jacob Lurie it is shown that the infinity category of quasi-coherent sheaves on a spectral Deligne-Mumford stack $\textnormal{Qcoh}_{\infty}(\A\times_S \dua(\A))$ contains the category usual abelian category of quasi-coherent sheaves on $\A\times_S \dua(\A)$. In other words, $\P_\A$ defines an object of $\textnormal{Qcoh}_{\infty}(\A\times_S \dua(\A))$. By 10.1.1 in spectral algebraic geometry (see also 6.2.4.1 for additional details) 
  \[\textnormal{Qcoh}_{\infty}(\A \times_S \dua(\A))\cong \lim_{(\A\times_S \dua(\A))(R)} \mod_R=\QC(\A\times_S \dua(\A)) \]
	\end{proof}
	
	\begin{theorem}
		We have an equivalence of categories 
		\[
		\QC(\A)\simeq \QC(\dua(\A))\]
	\end{theorem}
	
	\begin{proof}
		By \cite[Corollary 4.8]{bfn}, it suffices to show that $\QC(\G)$ and $\QC(\dua(\G))$  are dual as objects in 
		is the infinity category of stable infinity categories. This amounts to constructing 
		unit and counit functors as in the proof of \textit{loc. cit.} In what follows we denote by $p_{ij}$  and $p_i$ the projection onto the 
		$i,j$ components of $\G\times_S\dua(\G)\times \G$ and $i$ components respectively. We let $q_i$ the two projections of $\A\times_S \dua(\A)$ the two projections and $r_i\colon \dua(\A)\times_S \dua(\A)\ra \A $. Both the unit $u$ and counit $c$ are Fourier-Mukai functors induced by the kernel $P_\A$. We obtain 
		functors 
		\begin{align*}
			u\times 1:&\QC(\dua(\A))\longrightarrow \QC(\dua(\A)\times \A\times \dua(\A)) \\
			&F\longmapsto p_3^*(F)\otimes p_{12}^*\ \P_\A \\
			1\times c:& \QC(\dua(A)\times \A\times \dua(\A))\longrightarrow \QC(\dua(\A)) \\
			& G\longrightarrow p_{1*}(G\otimes p^*_{23}\P_\A).
		\end{align*}
We first claim that
\begin{equation}\label{eq:eq1}
p_{12}^*\P_\A \otimes p_{23}^* \P_\A\cong (m\times 1_{\A})^*\P_\A
\end{equation}
As $\P_\A\cong \eval_\A^*(\L_{\bG_M})$ this may be checked on $\dua(\bG_m)\times_S B\bG_m\times_S \dua(\bG_m)$ where it follows from $\dua(B\bG_m)=\bZ_S$ and the fact that $m_{B\bG_m}^*\L_{B\bG_m}\cong p_1^*\L_{B\bG_m}\otimes p_2^* \L_{B\bG_m}$ where $m$ is the multiplication $B\bG_m\times B\bG_m\ra B\bG_m$ and $p_i$ the two projections. We also have
\begin{equation}\label{eq:eq2}
p_{13*}(m\times 1_{\A})^*\P_\A\cong m_{\dua(\A)}^*q_{2*}\P_\A
\end{equation}
This follows from flat cohomology and base change applied to the diagram
\[\xymatrix{\dua(\A)\times_S \A\times_S \dua(\A)\ar[rr]^{m\times \iden_\A}\ar[d]^{p_{13}} & & \A\times_S \dua(A)\ar[d]_{q_2}\\ \dua(A)\times_S \dua(A)\ar[rr]_{m_{\dua(\A)}} & &\dua(\A)}\]
With this in hand the composition is (all functors are derived)
		\begin{align*}
			F \mapsto & p_{1*}(p_3^*F \otimes p_{12}^*\P_\A \otimes p_{23}^* \P_\A) \\
			\simeq & p_{1*}(p_3^* F \otimes (m\times 1_{\dua(\A)})^*\P_\A) \\
			\simeq & r_{1*}p_{13*}(p_{13}^*r_2^* F\otimes (m\times 1)^* \P_\A)\\
             \simeq & r_{1*}(r_2^* F\otimes p_{13*}(m\times 1)^* \P_\A)\\ 
			\simeq & r_{1*}( (r_2^* F)\otimes m^*q_{2*}\P_\A) \\
			\simeq & r_{1*}(r_2^* F \otimes m^*\epsilon_* \omega_{A^t/S}[-g]))
			\simeq (-1_{\dua(\A)})^*F\otimes \omega_{A^t/S}^{-1}[-g]
		\end{align*}
		giving the equivalence as needed.
	\end{proof}

\end{section}


\begin{thebibliography}{11}

\bibitem{TameStacks}
Abramovich, Dan. Olsson, Martin. Vistoli, Angelo
Tame stacks in positive characteristic
Ann. Inst. Fourier (Grenoble) Volume 58 (2008) no. 4 1057--1091

\bibitem{GoodModuli1} Alper, Jarod. Good moduli spaces for Artin stacks.
	Ann. Inst. Fourier (Grenoble)  63  (2013),  no. 6, 2349--2402.


\bibitem{MR1237825} Be\u{\i}linson, A. and Bernstein, J. A proof of Jantzen conjectures,
I. M. Gelfand Seminar, Adv. Soviet Math. 16, 1993


\bibitem{MR1437495} Behrend, K. and Fantechi, B.  The intrinsic normal cone.
 Invent. Math.  128  (1997),  no. 1, 45--88.


\bibitem{bfn}
	Ben-Zvi, David and  Francis, John and  Nadler, David. Integral transforms and Drinfeld centers in derived algebraic geometry.
	J. Amer. Math. Soc.  23  (2010),  no. 4, 909--966.

\bibitem{BLR} Bosch, Siegfried and Lütkebohmert, Werner and  Raynaud, Michel. Neron models.
	Ergebnisse der Mathematik und ihrer Grenzgebiete (3) [Results in Mathematics and Related Areas (3)], 21. Springer-Verlag, Berlin,  1990. {\rm x}+325 pp. ISBN: 3-540-50587-3 

 \bibitem{DModcharp}
Braverman, Alexander, Bezrukavnikov, Roman.
Geometric {L}anglands correspondence for D-modules in prime characteristic: the $\textnormal{GL}(n)$ case,
Pure Appl. Math. Q. Vol.3 (2007) no. 11, 153-179

\bibitem{Breen}
A Breen, Lawrence.
Extensions of abelian sheaves and Eilenberg-MacLane algebras
J Invent. Math. Volume 9 (1969/70) 15--44
   
\bibitem{Brochard1}  Brochard, Sylvain. Duality for commutative group stacks.
Int. Math. Res. Not. IMRN  2021,  no. 3, 2321--2388.
\bibitem{GeoLang}
   Chen, Tsao-Hsien, Zhu, Xinwen,
   Geometric Langlands in prime characteristic,
   Compos. Math. 153, (2017) 2, 395--452

\bibitem{MR498551} Deligne, Pierre. Théorie de Hodge. II.
(French)  Inst. Hautes Études Sci. Publ. Math.  No. 40  (1971), 5--57.

\bibitem{MR4085143} Derryberry, Richard. Stacky dualities for the moduli of Higgs bundles.
 Adv. Math.  368  (2020), 107152, 55 pp.

 
\bibitem{TFib}
Donagi, Ron and Pantev, Tony, With an appendix by Dmitry Arinkin.
    Torus fibrations, gerbes, and duality,
   Mem. Amer. Math. Soc. vol.193, (2008) 901


\bibitem{Linred}
 Hashimoto, Mitsuyasu.
Classification of the linearly reductive finite subgroup schemes of $SL_2$
J Acta Math. Vietnam. Vol.40 (2015) no.3 527--534
   

\bibitem{laumon} Laumon, Gerard. Transformation de Fourier generlisee. arXiv:alg-geom/9603004


 \bibitem{HAG} Lurie, Jacob. Higher algebra. Available at https://www.math.ias.edu/~lurie/

 \bibitem{Hall14} Hall, Jack Cohomology and base change for algebraic stacks,
  Math. Z. ,
  278,
  2014,
 401--429.

\bibitem{FM} Mukai, Shigeru . Duality between $D(X)$ and $D(\hat{X})$ with its application to Picard sheaves.
Nagoya Math. J.  81  (1981), 153--175.

\bibitem{relFM} Mukai, Shigeru. Fourier functor and its application to the moduli of bundles on an abelian variety.
Algebraic geometry, Sendai, 1985, 
515--550, Adv. Stud. Pure Math., 10, North-Holland, Amsterdam,  1987. 

\bibitem{hr} Hall, Jack.  Rydh, David. Perfect complexes in algebraic stacks.
Compos. Math.  153  (2017),  no. 11, 2318--2367.

\bibitem{MR1159232}
Laumon, Gerard.
 La transformation de Fourier géométrique et ses applications
 Proceedings of the International Congress of Mathematicians, Vol. I,
   II (Kyoto, 1990)
 437--445
1991



\bibitem{OortBook}
Oort, F.
Commutative group schemes
Lecture Notes in Mathematics
Vol.15 (1966)


\bibitem{stacks-project}
{The {stacks project authors}}.
{The Stacks project},
{\url{https://stacks.math.columbia.edu}},
{2023}.


\bibitem{SGA4Tome3}
Th\'{e}orie des topos et cohomologie \'{e}tale des sch\'{e}mas. {T}ome 3,
Lecture Notes in Mathematics, Vol. 305,
      Springer-Verlag, Berlin-New York, 1973


\bibitem{Wat75}
Waterhouse, William C. Basically bounded functors and flat sheaves, Pacific J. Math. vol.57 (1975), no.2, 597–610.








    
\end{thebibliography}
\end{document}